\documentclass{amsart}
\usepackage[english]{babel}
\usepackage{latexsym}
\usepackage{amssymb}
\usepackage{amsmath}
\usepackage{amscd}
\usepackage[matrix,arrow,curve]{xy}
\usepackage{graphicx}
\usepackage[cp1251]{inputenc}
\CompileMatrices 
 \DeclareMathOperator{\TrueIm}{Im}

 \DeclareMathOperator{\idd}{id}
 
\DeclareMathOperator{\repq}{Rep} \DeclareMathOperator{\HOMM}{Hom}

\DeclareMathOperator{\lspan}{span} \DeclareMathOperator{\Grass}{Gr}
\DeclareMathOperator{\IHom}{IHom} 
 \DeclareMathOperator{\GL}{GL}
\DeclareMathOperator{\Ende}{End} 
 
 \DeclareMathOperator{\Mat}{Mat}
\newtheorem{theorem}{Theorem}[section]
\newtheorem{lemma}[theorem]{Lemma}
\newtheorem{proposition}[theorem]{Proposition}
\newtheorem{corollary}[theorem]{Corollary}

\theoremstyle{definition}

\newtheorem{example}[theorem]{Example}

\theoremstyle{remark}
\newtheorem{remark}[theorem]{Remark}

\numberwithin{equation}{section}

\begin{document}
\fontfamily{ptm} \fontsize{11pt}{14pt} \selectfont
\author[S. Fedotov]{Stanislav Fedotov}
\title[Framed moduli spaces and tuples of operators]{Framed moduli spaces and tuples of operators}
\address{Moscow State University, department of Higher Algebra, Russia, 119991, Moscow, GSP-1, 1 Leninskiye Gory, Main Building
\endgraf Email: glwrath@yandex.ru}
\sloppy

\maketitle

\begin{abstract} In this work we address the classical problem of classifying tuples of linear operators and linear functions on a finite dimensional vector space up to base change. Having adopted for the situation considered a construction of framed moduli spaces of quivers, we develop an explicit classification of tuples belonging to a Zariski open subset. For such tuples we provide a finite family of normal forms and a procedure allowing to determine whether two tuples are equivalent.
\end{abstract}

\section*{Introduction}

A quiver $Q$ is a diagram of arrows, determined by two finite sets $Q_0$ (the
set of ``vertices'') and $Q_1$ (the set of ``arrows'') with two maps
$h,t: Q_1\rightarrow Q_0$ which indicate the vertices at the head
and tail of each arrow. A representation $W$ of $Q$
is a collection of (probably infinite dimensional) $\Bbbk$-vector spaces
$W_i$, for each $i\in Q_0$, together with linear maps $W_a:
W_{ta}\rightarrow W_{ha}$, for each $a\in Q_1$. The dimension vector
$\alpha\in\mathbb{Z}^{Q_0}$ of such a representation is given by
$\alpha_i = \dim_{\Bbbk}{W_i}$. A morphism $\psi:W\rightarrow U$ of representations consists of
linear maps $\psi_i: W_i\rightarrow U_i$, for each $i\in Q_0$, such
that $\psi_{ha}W_a = U_a\psi_{ta}$, for each $a\in Q_1$.
It is an isomorphism if and only if each $\psi_i$ is.


Having chosen vector spaces $W_i$ of dimension $\alpha_i$, the
isomorphism classes of representations of $Q$ with dimension vector
$\alpha$ are in natural one-to-one correspondence with the orbits of
the group
$$GL(\alpha) := \prod_{i\in Q_0}GL(W_i)$$in
the representation space
$$\repq(Q,\alpha) := \bigoplus_{a\in Q_1}\HOMM(W_{ta},W_{ha}).$$

The action is given by
$(g\cdot W)_a~=~g_{ha}W_ag^{-1}_{ta}$, where $g =
(g_i)_{i\in Q_0}\in GL(\alpha)$. Note that the one-parameter
subgroup $\Delta = \{(tE,\ldots,tE)\}$ acts trivially.

Quivers provide a convenient interpretation of many classical problems of linear algebra. The one we are particularly interested in is classification of tuples of $q$ linear operators and $k$ linear functions on an $m$-dimensional vector space. In the language of quivers this is equivalent to classification of $(m,1)$-dimensional representations of $L_{q,k}$, where by $L_{q,k}$ we denote the quiver with two vertices, $q$ loops in the first vertex and $k$ more arrows going from the first vertex to the second one. This problem is known to be wild even for $q = 2$ and $k = 0$, that is no hope remains to write down a complete list of isomorphism classes of representations or even to obtain an algorithm determining whether two given representations are isomorphic. In fact, representation theory of the quiver $L_2 := L_{2,0}$ is proved to be undecidable, see~\cite{WB} and~\cite{KM} for a rigorous formulation and a proof of this result.

However, for $\alpha = (m,1)$ it is possible to explicitly classify representations belonging to a Zariski open subset of $\repq(L_{q,k},\alpha)$. The idea comes from the study of stable framed representations.

Let $Q$ be a quiver and $\alpha$ be a dimension vector. Fix an
additional dimension vector $\zeta$ and consider the space $\repq(Q,
\alpha, \zeta) := \repq(Q, \alpha)\oplus\bigoplus_{i\in
Q_0}\HOMM_{\Bbbk}(\Bbbk^{\alpha_i},\Bbbk^{\zeta_i})$. Its elements
are said to be {\it framed representations} of $Q$. Define a
$GL(\alpha)$-action on $\repq(Q, \alpha, \zeta)$ by $g\cdot(M,
(f_i)_{i\in Q_0}) = (g\cdot M, (f_ig_i^{-1})_{i\in Q_0})$. A framed representation $(M, (f_i)_{i\in Q_0})$ is called {\it stable} if there is no
nonzero subrepresentation $N$ of $M$ with 
$M_i\subseteq\ker{f_i}$, for all $i\in Q_0$. Denote by $\repq^s(Q, \alpha, \zeta)$ the set of stable
framed representations. It is known (see, for example,~\cite[Theorem 2.3]{Reineke}), that the subset $\repq^s(Q,\alpha,\zeta)$ admits a geometric quotient, i.e., a morphism to an algebraic variety $\mathcal{M}^s(Q,\alpha,\zeta):=\repq^s(Q,\alpha,\zeta)/\!\!/GL(\alpha)$ whose fibers coincide with $GL(\alpha)$-orbits. Moreover, for quivers without oriented cycles M. Reineke proved~\cite[Proposition~3.9]{Reineke} that the quotient space $\mathcal{M}^s(Q,\alpha,\zeta)$ is isomorphic to a Grassmannian of subrepresentations of a certain injective representation of $Q$. In the general case the quotient is not projective and may not be realized as a Grassmannian of subrepresentations. However, some geometric structure may be revealed by projecting $\mathcal{M}^s(Q,\alpha,\zeta)$ to the categorical quotient $\repq(Q,\alpha,\zeta)/\!\!/\GL(\alpha)$ and studying fibers of this projection, see~\cite{smooth} and~\cite{SN} for details.

From~\cite[Proposition 0.9]{GIT} it follows that the quotient morphism $\repq^s(Q,\alpha,\zeta)\rightarrow\mathcal{M}^s(Q,\alpha,\zeta)$ is a principal fiber bundle. It remains a problem, however, to explicitly describe a finite (and possibly minimal) trivializing covering of the quotient. We construct such a covering using $J$-skeleta of framed representations, a concept that is a version of the one introduced by K. Bongartz and B.~Huisgen-Zimmermann for representations of finite dimensional algebras (see, for example,~\cite{BH-cl}) adopted and partially simplified to fit our setup. Namely, we show (Theorem 3.3) that $\repq^s(Q,\alpha,\zeta) = \bigcup_{\mathfrak{S}}X(\mathfrak{S})$, where $X(\mathfrak{S})$ are open subsets parameterized by $J$-skeleta $\mathfrak{S}$ such that $X(\mathfrak{S})\cong\GL(\alpha)\times\mathbb{A}^N$, for some positive integer $N$, and the restriction of the quotient map to $X(\mathfrak{S})$ is the projection onto the second factor.

Framed representations admit another useful interpretation. 
Consider a new quiver $Q^{\zeta}$ with $Q^{\zeta}_0 =
Q_0\cup\{\infty\}$, the arrow of $Q^{\zeta}$ being those of $Q$
together with $\zeta_i$ arrows from $i$ ($i\in Q_0$) to $\infty$. We
also extend the dimension vector $\alpha$ to $\alpha^{\zeta}$,
setting $\alpha^{\zeta}_i = \alpha_i$ for $i = 1,\ldots,n$ and
$\alpha^{\zeta}_{\infty} = 1$. It is easy to show that $\repq(Q,\alpha,\zeta)$ may be identified with $\repq(Q^{\zeta},\alpha^{\zeta})$.

Clearly $\repq(L_{q,k},(m,1))$ is the same as $\repq(L_q^{(k)},m^{(k)})$, i.e., as $\repq(L_q,m,k)$. So, there is a Zariski open subset $\repq^s(L_{q,k},(m,1))$ of $\repq(L_{q,k},(m,1))$ where a complete classification of representations is possible. Translating our definition of stable pairs into the language of linear algebra, we may say that $\repq^s(L_{q,k},(m,1))$ consists of such tuples $(\varphi_1,\ldots,\varphi_q,f_1,\ldots,f_k)\in\left(\Ende_{\Bbbk}(\Bbbk^m)\right)^q\oplus\left((\Bbbk^m)^{\ast}\right)^k$ that no common proper nonzero invariant subspace of $\varphi_i$ lies in the common kernel of all~$f_j$. In this paper we show how an explicit classification may be obtained in this setup over an arbitrary field~$\Bbbk$.

Section 1 is devoted to exploring a generalized version of the construction introduced in~\cite{Reineke}. In Section 2 we define $J$-skeleta of framed representations and show their existence. In Section 3 we prove that the quotient may be embedded as a locally closed subset in a product of ordinary Grassmannians and construct the above mentioned trivializing covering of $\mathcal{M}^s(Q,\alpha,\zeta)$. Furthermore, we provide a finite family of normal forms for each stable pair and an algorithm allowing to determine whether two stable framed representations are isomorphic. In Section 4 we give a series of examples illustrating how this technique works.

The author thanks his supervisor I. Arzhantsev for useful discussions.

\section{Stable framed representations}

Let $Q$ be a quiver with $n$ vertices, $\alpha$ and $\zeta$ be two dimension vectors. Choose a vector space $V = \bigoplus_{i\in Q_0}V_i$ with $\dim{V_i} = \zeta_i$, for $i\in Q_0$. 
Elements of $\repq(Q,\alpha,\zeta)$ may be viewed as pairs $(M,f)$, where $M$ is a representation of $Q$ and $f = \left(f_i: M_i\rightarrow V_i\right)_{i\in Q_0}$ is a map of graded vector spaces.

Recall that a path in $Q$ is a formal product of arrows $a_1\cdot\ldots\cdot a_k$ such that $t(a_i) = h(a_{i+1})$, for all $i = 1,\ldots,k-1$, or a symbol $e_i$ with $i\in Q_0$. For a path $\tau = a_1\cdot\ldots\cdot a_k$ we set $t(\tau) = t(a_k)$ and $h(\tau) = h(a_1)$. We also put $h(e_i) = t(e_i) = i$. There is an obvious way to multiply successive paths: if $h(\tau) = t(\sigma)$, the product $\sigma\cdot\tau$ is defined as the concatenation of these paths. All $e_i$ are treated as paths of zero length, that is $e_i^2 = e_i$, for all $i\in Q_0$, and $\tau e_{t(\tau)} = e_{h(\tau)}\tau$, for every path $\tau$.

For each $i\in Q_0$ denote by $I_i$
the following representation of $Q$. Set
$$(I_i)_j =
\left(\lspan\left\{\tau\mid\mbox{$\tau: j\rightsquigarrow i$ is a path in $Q$}\right\}\right)^*,$$ where ``$\tau: j\rightsquigarrow i$'' means that $\tau$ starts
in the $j$-th vertex and ends in the $i$-th one and ${(\cdot)}^*$ stands for the dual vector space; in this case
$((I_i)_{a: k\rightarrow l}f)(\tau) = f(\tau a)$, where
$\tau: l\rightsquigarrow i$. This may be rewritten in a more convenient way using the elements in $(I_i)_j$ dual to paths. Namely, to each path $\tau: j\rightsquigarrow i$ in $Q$ we associate an element $\tau^*$ in $(I_i)_j$ such that for every $\sigma: j\rightsquigarrow i$ we have
$$\tau^*(\sigma) = \begin{cases}
1,\mbox{ if $\tau = \sigma$},\\
0,\mbox{ otherwise}.
\end{cases}$$
Observe that elements of $(I_i)_j$ may be written as (probably infinite) formal series in $\tau^*$, for $\tau: i\rightsquigarrow j$. In this notation the maps $(I_i)_a$, $a\in Q_1$, are as follows:
$$(I_i)_a(\tau^*) =
\begin{cases}
\lambda^*,\mbox{ if $\tau = \lambda a$},\\
0,\mbox{ otherwise}.
\end{cases}$$

Consider the representation $J :=
\bigoplus_{i\in Q_0}I_i\otimes_{\Bbbk} V_i$. Notice that as a
$\Bbbk$-linear space
\begin{equation*}J_i = e_iJ\cong\prod_{j\in
Q_0}(I_j)_i\otimes_{\Bbbk}V_j\cong \prod_{j\in
Q_0}\prod_{\tau:i\rightsquigarrow
j}V_j\cong\prod_{\tau:i\rightsquigarrow j}V_j.\end{equation*}
Sometimes it is convenient to label each component $V_j$ by the corresponding path $\tau$ writing $J_i\cong\prod_{\tau:i\rightsquigarrow j}V_j^{(\tau)}$.

%
%

Given a point $(M, f)\in\repq(Q,\alpha,\zeta)$ we define a map
$\Phi_{(M,f)} = (\varphi_i)_{i\in Q_0}: M\rightarrow J$ by the
following rule: \begin{equation}\varphi_i = \prod_{\tau:
i\rightsquigarrow j} f_j\tau:
M_i\rightarrow\prod_{\tau:i\rightsquigarrow j}V_j,\end{equation}
where $\tau(x) :=
M_{a_1}\ldots M_{a_k}(x)$ for $x\in M_i$ and $\tau = a_1\ldots a_k$.

\smallskip

The following lemma is straightforward.

\begin{lemma} The map $\Phi_{(M,f)}$ is a morphism of representations of $Q$.
\end{lemma}

\begin{proposition} The subspace $\ker\Phi_{(M,f)} = \oplus_{i\in
Q_0}\ker\varphi_i$ is the maximal $\Bbbk Q$-submodule of $M$
contained in $\ker{f}$.
\end{proposition}

\begin{proof} It follows from Lemma 1.1 that $\ker\Phi_{(M,f)}$ is a
$\Bbbk Q$-submodule of $M$. One also easily observes that $\ker{\Phi_{(M,f)}}\subseteq\ker{f}$. Now, let $U$ be a $\Bbbk Q$-submodule of
$M$ contained in $\ker{f}$. For each $\tau:i\rightsquigarrow j$ we
then have $\tau U_i = \tau e_iU = \tau U = e_j\tau U \subseteq U_j$.
This implies that $f_j(\tau\cdot x) = 0$, for all $x\in U, j\in Q_0$
and for all paths $\tau$, i. e. $U\subseteq\ker\Phi_{(M,f)}$.
\end{proof}

\begin{corollary} The map $\Phi_{(M,f)}:M\rightarrow J$ is injective if and only if
the pair $(M,f)$ is stable.
\end{corollary}

This observation is crucial for the construction. Associated to the maps
$\varphi_i$ are (probably infinite) matrices with rows
$f_{iq}\tau$, where $\tau: j\rightsquigarrow i$, $q\in\{1,\ldots,\zeta_i\}$, and $f_{iq}$ stands for the $q$-th row of the matrix of $f_i$.
The map is injective if and only if one of
$\alpha_i\times\alpha_i$ minors of its matrix is nonzero, i.e., for some
$\tau_1,\ldots,\tau_{\alpha_i}$ and $q_1,\ldots,q_{\alpha_i}$, we have
$$\det{
\begin{pmatrix}
f_{iq_1}\tau_1\\
f_{iq_2}\tau_2\\
\vdots\\
f_{iq_{\alpha_i}}\tau_{\alpha_i}
\end{pmatrix} }\ne 0.$$
Therefore, a pair $(M, f)$ is stable if and only if for each $i\in
Q_0$ there is a set of numbers $q^{(i)}_1,\ldots,q^{(i)}_{\alpha_i}$ and a set of distinct paths
$\tau^{(i)}_1,\ldots,\tau^{(i)}_{\alpha_i}$ with
$$D^{(q^{(1)}_1,\ldots,q^{(1)}_{\alpha_1};\ldots;
q^{(n)}_1,\ldots,q^{(n)}_{\alpha_n})}_{(\tau^{(1)}_1,\ldots,\tau^{(1)}_{\alpha_1};\ldots;
\tau^{(n)}_1,\ldots,\tau^{(n)}_{\alpha_n})} := \det{
\begin{pmatrix}
f_{iq^{(1)}_1}\tau^{(1)}_1\\
f_{iq^{(1)}_2}\tau^{(1)}_2\\
\vdots\\
f_{iq^{(1)}_{\alpha_1}}\tau^{(1)}_{\alpha_1}
\end{pmatrix} }\cdot\ldots\cdot\det{
\begin{pmatrix}
f_{iq^{(n)}_1}\tau^{(n)}_1\\
f_{iq^{(n)}_2}\tau^{(n)}_2\\
\vdots\\
f_{iq^{(n)}_{\alpha_n}}\tau^{(n)}_{\alpha_n}
\end{pmatrix} }\ne 0.$$
Thus $\repq^s(Q, \alpha, \zeta)$ is covered by open
subsets $U^{(q^{(1)}_1,\ldots,q^{(1)}_{\alpha_1};\ldots;
q^{(n)}_1,\ldots,q^{(n)}_{\alpha_n})}_{(\tau^{(1)}_1,\ldots,\tau^{(1)}_{\alpha_1};\ldots;
\tau^{(n)}_1,\ldots,\tau^{(n)}_{\alpha_n})}$, which are the nonzero
loci of the corresponding
$D^{(q^{(1)}_1,\ldots,q^{(1)}_{\alpha_1};\ldots;
q^{(n)}_1,\ldots,q^{(n)}_{\alpha_n})}_{(\tau^{(1)}_1,\ldots,\tau^{(1)}_{\alpha_1};\ldots;
\tau^{(n)}_1,\ldots,\tau^{(n)}_{\alpha_n})}$, and consequently the
moduli space is covered by the quotients
$V^{(q^{(1)}_1,\ldots,q^{(1)}_{\alpha_1};\ldots;
q^{(n)}_1,\ldots,q^{(n)}_{\alpha_n})}_{(\tau^{(1)}_1,\ldots,\tau^{(1)}_{\alpha_1};\ldots;
\tau^{(n)}_1,\ldots,\tau^{(n)}_{\alpha_n})} :=
U^{(q^{(1)}_1,\ldots,q^{(1)}_{\alpha_1};\ldots;
q^{(n)}_1,\ldots,q^{(n)}_{\alpha_n})}_{(\tau^{(1)}_1,\ldots,\tau^{(1)}_{\alpha_1};\ldots;
\tau^{(n)}_1,\ldots,\tau^{(n)}_{\alpha_n})}/\!\!/GL(\alpha)$. In next section we shall prove that this covering admits a finite subcovering.

\section{Skeleta of stable pairs}

Let $Q$, $\alpha$ and $\zeta$ be as before. Consider a quiver $Q^{\zeta}$ with $Q^{\zeta}_0 =
Q_0\cup\{\infty\}$, the arrows of $Q^{\zeta}$ being those of $Q$
together with $\zeta_i$ arrows from each $i\in Q_0$ to $\infty$. Denote the new arrows by $f_{iq}$, where $i$ indicates the tail of an arrow and $q\in\{1,\ldots,\zeta_i\}$. We
also extend the dimension vector $\alpha$ to $\alpha^{\zeta}$,
setting $\alpha^{\zeta}_i = \alpha_i$ for $i = 1,\ldots,n$ and
$\alpha^{\zeta}_{\infty} = 1$.

Observe that the sets $\repq(Q,\alpha,\zeta)$ and $\repq(Q^{\zeta},\alpha^{\zeta})$ may be identified in a $\GL(\alpha)$-invariant way. Indeed,
$$\repq(Q,\alpha,\zeta)=\repq(Q,\alpha)\oplus\bigoplus_{i\in Q_0}\HOMM_{\Bbbk}(\Bbbk^{\alpha_i},\Bbbk^{\zeta_i})\cong$$
$$\cong
\repq(Q,\alpha)\oplus\bigoplus_{i\in Q_0}\HOMM_{\Bbbk}(\Bbbk^{\alpha_i},\Bbbk)^{\zeta_i}=\repq(Q^{\zeta},\alpha^{\zeta}).$$
In terms of matrices this isomorphism has the following interpretation. Let $(M,f)$ be a framed representation and $\widetilde{M}$ be the corresponding representation of $Q^{\zeta}$. Then matrices of $\widetilde{M}_{f_{iq}}$, $q = 1,\ldots,\zeta_i$ are rows of the matrix of $f_i$ (i.e., what was denoted by $f_{iq}$ in Section 1). This justifies our seeming abuse of notation.

\begin{example} {\rm Let $Q = \xymatrix{1\ar@/_1pc/[rr]_a\ar[r]^b & 2\ar[r]^c & 3}$
and $\zeta = (1,2,3)$. Then
we have
$$
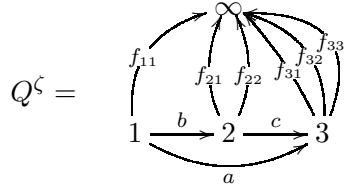
\xymatrix{{}\save[]-<0cm, 0.6cm>*\txt{$Q^{\zeta} =\ $}\restore & & {}\save[]+<0cm, 0.5cm>*\txt{$\infty$}\restore &\\
&1\ar@/_1pc/[rr]_a\ar[r]^b\ar@/^1pc/[ur]+<-0.3cm,0.44cm>|{f_{11}} & 2\ar[r]^c\ar@/^0.5pc/[u]+<-0.1cm,0.4cm>|{f_{21}}\ar@/_0.5pc/[u]+<0.1cm,0.4cm>|{f_{22}} & 3\ar@/_0.2pc/[ul]+<0.2cm,0.44cm>|{f_{31}}\ar@/_1pc/[ul]+<0.2cm,0.44cm>|{f_{32}}
\ar@/_1.8pc/[ul]+<0.2cm,0.44cm>|{f_{33}}}$$
Furthermore, if $\alpha = (2,2,1)$ and
$$\xymatrix{{}\save[]-<0cm, 0.6cm>*\txt{$(M,f) =\ $}\restore & & \Bbbk & & \Bbbk^2 & & \Bbbk^3\\
&&\Bbbk^2\ar@/_1pc/[rrrr]|{M_a}\ar[rr]|{M_b}\ar[u]^{\left(\begin{smallmatrix}1 & 5\end{smallmatrix}\right)} & & \Bbbk^2\ar[rr]|{M_c}\ar[u]^{\left(\begin{smallmatrix}3 & 1\\ 4& 2\end{smallmatrix}\right)} & & \Bbbk
\ar[u]^{\left(\begin{smallmatrix}3\\ 0\\ 1\end{smallmatrix}\right)}}$$
for some $M_a$, $M_b$ and $M_c$, then the corresponding representation $\widetilde{M}$ of $Q^{\zeta}$ is as follows
$$
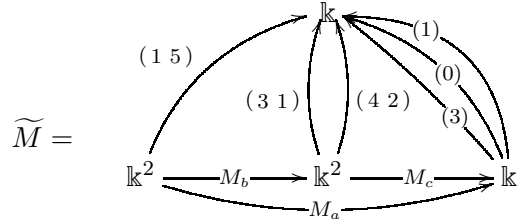
\xymatrix{{}\save[]-<0cm, 0.6cm>*\txt{$\widetilde{M} =\ $}\restore & && {}\save[]+<0cm, 1.0cm>*\txt{$\Bbbk$}\restore &&\\
&\Bbbk^2\ar@/_1pc/[rrrr]|{M_a}\ar[rr]|{M_b}\ar@/^1pc/[urr]+<-0.3cm,0.94cm>^{\left(\begin{smallmatrix}1 & 5\end{smallmatrix}\right)} && \Bbbk^2\ar[rr]|{M_c}\ar@/^0.5pc/[u]+<-0.1cm,0.9cm>^{\left(\begin{smallmatrix}3 & 1\end{smallmatrix}\right)}\ar@/_0.5pc/[u]+<0.1cm,0.9cm>_{\left(\begin{smallmatrix}4 & 2\end{smallmatrix}\right)} && \Bbbk\ar@/_0.2pc/[ull]+<0.2cm,0.94cm>|(.35){(3)}\ar@/_1pc/[ull]+<0.2cm,0.94cm>|{(0)}
\ar@/_1.8pc/[ull]+<0.2cm,0.94cm>|(0.7){(1)}}$$}
\end{example}

The representation $J$ may also be extended in a natural way to a representation $\widetilde{J}$ of $Q^{\zeta}$. Set $\widetilde{J}_i = J_i$ and $\widetilde{J}_a = J_a$, for $i\in Q_0$ and $a\in Q_1$. Set further $\widetilde{J}_{\infty} = \Bbbk$ and $\bigoplus_{b: i\rightarrow\infty}J_b: J_i\rightarrow\Bbbk^{\zeta_i}$ be the projection $\prod_{\tau:i\rightsquigarrow j}V_j^{(\tau)}\rightarrow V_i^{(e_i)}$, for each $i\in Q_0$. A straightforward calculation shows that thus constructed $\widetilde{J}$ is isomorphic to the representation $I_{\infty}$ of $Q^{\zeta}$. In particular, elements of $\widetilde{J}_i$ may be represented as (possibly infinite) formal series in $(f_{jq}\tau)^*$, for $j\in Q_0$, $q = 1,\ldots,\zeta_j$, and $\tau:i\rightsquigarrow j$. This implies that paths in $Q^{\zeta}$ may be viewed as linear functions on $\widetilde{J}$.

For a framed representation $(M,f)\in\repq(Q,\alpha,\zeta)$ consider the corresponding representation $\widetilde{M}$ of $Q^{\zeta}$. The map $\Phi_{(M,f)}$ induces then a morphism $\widetilde{\Phi}_{\widetilde{M}}: \widetilde{M}\rightarrow\widetilde{J}$ of representations of $Q^{\zeta}$ defined by $\widetilde{\varphi}_{i} = \varphi_i$, for $i\in Q_0$, $\widetilde{\varphi}_{\infty} = \idd_{\Bbbk}$. It follows from Corollary 1 that a pair $(M,f)$ is stable if and only if the corresponding map $\widetilde{\Phi}_{\widetilde{M}}$ is an embedding.


\medskip

{\bf Definition.} By a {\it $J$-skeleton of a stable pair $(M, f)$} we understand a
set $\mathfrak{S}$ of paths of nonzero length in $Q^{\zeta}$ ending
in $\infty$ with the following properties:
\begin{itemize}
\item[(1)] Restrictions of paths in $\mathfrak{S}$ together with $e_\infty$ give a basis
in $\widetilde{\Phi}_{\widetilde{M}}\left(\widetilde{M}\right)^*$.
\item[(2)] Whenever $\tau a$ is in $\mathfrak{S}$ and $\tau\ne e_{\infty}$,
$\tau$ is also in $\mathfrak{S}$.
\end{itemize}
The {\it dimension vector} of a $J$-skeleton $\mathfrak{S}$ is the
dimension vector $\alpha = \underline{\dim}(\mathfrak{S})$ of any
stable pair with $J$-skeleton $\mathfrak{S}$. It is easy to see that
$\alpha_i$ equals the number of paths in $\mathfrak{S}$ ending on
$f_{iq}$ for any $q = 1,\ldots,\zeta_i$. For $i\in Q_0$ we set
$\mathfrak{S}_i = \left\{f_{iq}\tau\in\mathfrak{S}\right\}$.

\begin{lemma} Every stable pair $(M,f)$ has a $J$-skeleton.
\end{lemma}

\begin{proof} Let $(M,f)$ be a stable pair. Let also $\widetilde{M}$ be the corresponding representation of $Q^{\zeta}$. Denote by $N$ its image
$\widetilde{\Phi}_{\widetilde{M}}(\widetilde{M})\subseteq\widetilde{J}$. First of all, we need to show that restrictions of
paths in $Q$ generate $\widetilde{\Phi}_{\widetilde{M}}(\widetilde{M})^*$ as a vector space. Let $\varpi_1,\ldots,\varpi_m$ be a basis of $\widetilde{\Phi}_{\widetilde{M}}(\widetilde{M})$, where $\varpi_i = \sum c_{iq,\tau}(f_{iq}\tau)^*$. Observe that, for some $t$,
$$\dim\left(\lspan\left\{\sum_{f_{i_q}\tau,\ l(\tau)\leqslant t}c_{i_q,\tau}(f_{iq}\tau)^*\right\}\right) = m,$$
where $l(\tau)$ stands for the length of a path $\tau$. Then, obviously, the linear span of all paths of length not greater than $t$ generates $\widetilde{\Phi}_{\widetilde{M}}(\widetilde{M})^*$.

We see now, that restrictions of paths in $Q$ give a basis of $\Phi_{\widetilde{M}}(\widetilde{M})^*$, but less evident
is condition (2).

 We construct its $J$-skeleton inductively.
We start by taking, for each $i\in Q_0$, a maximal tuple
$f_{iq_1},\ldots,f_{iq_t}$ with $f_{iq_1}|_N,\ldots,f_{iq_t}|_N$
linearly independent. Further, on each step we add a path $\tau a$,
where $\tau$ is one of the paths we added before, $a$ is an arrow,
and $\tau a|_{N}$ does not lie in the linear span of restrictions of all the preceding
paths. We proceed until we obtain a maximal linearly independent
set of restrictions $\Gamma$. It should be proved, however, that it
is a basis of $N^*$. Let
$\tau|_{N}\notin\lspan\left\{\Gamma\right\}$. If none of its final
subpaths restricted to $N$ is in $\lspan\left\{\Gamma\right\}$, then
we have found $f_{lq_l}$ whose restriction is not in
$\lspan\left\{\Gamma\right\}$, a contradiction. Thus $\tau =
\mu\nu$, where $\mu|_N\in\lspan\left\{\Gamma\right\}$, i.e., $\mu|_N =
\sum_{\kappa\in\Gamma}c_{\kappa}\kappa|_N$. Consequently, $\tau|_{N}
= \sum_{\kappa\in\Gamma}c_{\kappa}\kappa|_N\nu|_N$. But by
maximality of $\Gamma$ each $\kappa|_N\nu|_N$ is in
$\lspan\left\{\Gamma\right\}$.
\end{proof}

Obviously, $\repq^s(Q, \alpha,\zeta) =
\bigcup_{\mathfrak{S}}\repq(Q, \mathfrak{S})$, where
$$\repq(Q,
\mathfrak{S}) = \left\{(M,f)\in\repq^s(Q^{\zeta},
\alpha,\zeta)\mid\mbox{$(M,f)$ has $J$-skeleton
$\mathfrak{S}$}\right\}$$ and $\mathfrak{S}$ run through all
possible $J$-skeleta for dimension vectors $\alpha$ and $\zeta$.
Now note that if $\mathfrak{S} =
\left\{f_{1q^{(1)}_{1}}\tau^{(1)}_1,\ldots,f^{(1)}_{q_{1\alpha_1}}\tau^{(1)}_{\alpha_1},\ldots,
f_{nq^{(n)}_{1}}\tau^{(n)}_1,\ldots,f^{(n)}_{q_{n\alpha_n}}\tau^{(n)}_{\alpha_n}\right\}$,
then $\repq(Q, \mathfrak{S})$ is exactly the open subset
$U^{(q^{(1)}_1,\ldots,q^{(1)}_{\alpha_1};\ldots;
q^{(n)}_1,\ldots,q^{(n)}_{\alpha_n})}_{(\tau^{(1)}_1,\ldots,\tau^{(1)}_{\alpha_1};\ldots;
\tau^{(n)}_1,\ldots,\tau^{(n)}_{\alpha_n})}$. Observe that there are only
finitely many $J$-skeleta, since lengths of paths in an $\alpha$-dimensional $J$-skeleton are bounded by $\max_i\alpha_i$. Therefore, we obtain a finite covering of $\repq^s(Q,\alpha,\zeta)$ by subsets of form $U^{(q^{(1)}_1,\ldots,q^{(1)}_{\alpha_1};\ldots;
q^{(n)}_1,\ldots,q^{(n)}_{\alpha_n})}_{(\tau^{(1)}_1,\ldots,\tau^{(1)}_{\alpha_1};\ldots;
\tau^{(n)}_1,\ldots,\tau^{(n)}_{\alpha_n})}$.

\section{Embedding of the moduli space}

Let $\Gamma(\alpha)$ be the set of all paths occurring in
$J$-skeleta with dimension vector $\alpha$ and
$\widetilde{\Gamma}(\alpha)$ be the union of $\Gamma$ with
$\left\{\tau a\mid\tau\in\Gamma(\alpha), a\in Q_1, h(a) = t(\tau)\right\}$. Let also
$\widehat{J} =
\bigoplus_{\tau\in\widetilde{\Gamma}(\alpha)}V_{h(\tau)}^{(\tau)}$.
Note that $\widehat{J}$ has a natural $Q_0$-grading.
Indeed, one may set $\widehat{J}_i =
\bigoplus_{\tau\in\widetilde{\Gamma}(\alpha), t(\tau) =
i}V_{h(\tau)}^{(\tau)}$. Consider the map $\widehat{\Phi}_{(M,f)}:
M\rightarrow\widehat{J}$ defined as follows:
$$\widehat{\varphi}_i =
\bigoplus_{\tau\in\widetilde{\Gamma}(\alpha)}f_{h(\tau)}\tau :
M_i\rightarrow\bigoplus_{\tau\in\widetilde{\Gamma}(\alpha), t(\tau)
= i}V_{h(\tau)}^{(\tau)} = \widehat{J}_i.$$ It is obvious that a
pair $(M,f)$ is stable if and only if $\widehat{\Phi}_{(M,f)}$ is
injective. The advantage of $\widehat{\Phi}_{(M,f)}$ is that it
maps to a finite dimensional vector space.

We now need to fix notation that will be used throughout the rest of the paper. Let
$\widetilde{\Gamma}_i(\alpha)$ be a subset of $\widetilde{\Gamma}(\alpha)$ consisting of paths starting at $i$. Let further  $(B^{(1)},\ldots,B^{(n)})\in\prod_{i =
1}^n\Mat_{\left|\widetilde{\Gamma}_i(\alpha)\right|\times
\alpha_i}(\Bbbk)$. By definition of $\widetilde{\Gamma}(\alpha)$ rows of $B^{(i)}$ correspond to paths in $Q^{\zeta}$ starting at $i$. So, instead of using numerical indices we denote, for a path $\tau$ in $Q$, the row of $B^{(t(\tau))}$ corresponding to
$\tau$ by $B_{\tau}$.
Let, furthermore, $\mathfrak{T}$ be a subset of $\widetilde{\Gamma}(\alpha)$ (not necessary a skeleton). Denote by $\mathfrak{T}_i$ the subset in $\mathfrak{T}$ consisting of paths starting at $i$. Set $\underline{\dim}(\mathfrak{T}) = (|\mathfrak{T}_1|, \ldots,|\mathfrak{T}_n|)$. Finally, for a collection $\mathfrak{T}$ with dimensional vector $\alpha$ let $B(\mathfrak{T}_i)$ be the submatrix of $B^{(i)}$ composed of
rows corresponding to paths in $\mathfrak{T}_i$ 
and $U(\mathfrak{T})$ be the open subset in $\prod_{i = 1}^n\Mat_{\left|\widetilde{\Gamma}_i(\alpha)\right|\times \alpha_i}(\Bbbk)$ consisting of tuples $(B^{(1)},\ldots,B^{(n)})$ with all $B(\mathfrak{T}_i)$ nondegenerate.

For a $Q_0$-graded space $W = \bigoplus_{i\in Q_0}W_i$ define $\IHom_{\alpha}(W) := \prod_{i\in Q_0}\IHom_{\alpha_i}(W_i)$, where $\IHom_{\alpha_i}(W_i)$ is the set
of all injective linear maps from the vector space $\Bbbk^{\alpha_i}$ to $W_i$. We also define
$\Grass_{\alpha}(W)$ as the product of Grassmannians $\prod_{i\in
Q_0}\Grass_{\alpha_i}(W_i)$. It is easy to see that
$\Grass_{\alpha}(W)$ is a quotient of $\IHom_{\alpha}(W)$ by the
natural action of $GL(\alpha)$. We now introduce a map
$$\widehat{\Phi}: \repq^s(Q,\alpha,\zeta)\rightarrow\IHom_{\alpha}(\widehat{J}),\quad
(M,f)\mapsto\widehat{\Phi}_{(M,f)}.$$

Identify $\IHom_{\alpha}(\widehat{J})$ with an open subset
$$\bigcup_{\mathfrak{T}:\,\underline{\dim}{\mathfrak{T}} = \alpha}U(\mathfrak{T})\subseteq\prod_{i =
1}^n\Mat_{\left|\widetilde{\Gamma}_i(\alpha)\right|\times
\alpha_i}(\Bbbk).$$
It is easy to see that $\TrueIm\widehat{\Phi}$ is contained in
$$Z(\alpha) := \bigcup_{\begin{smallmatrix}\mbox{\small$\mathfrak{S}$ is a $J$-skeleton}\\
\underline{\dim}(\mathfrak{S}) = \alpha
\end{smallmatrix}}U(\mathfrak{S}).$$

{\bf Definition.} For a $J$-skeleton $\mathfrak{S}$ set $X(\mathfrak{S}) = \TrueIm(\widehat{\Phi})\cap U(\mathfrak{S})$.

\begin{proposition} The image of $\widehat{\Phi}$ is a locally closed subvariety in
$\IHom_{\alpha}(\widehat{J})$.
\end{proposition}

\begin{proof} It is sufficient to show that each $X(\mathfrak{S})$ is closed in $U(\mathfrak{S})$. Fix a $J$-skeleton $\mathfrak{S}$. For an arrow $a\in Q_1$ set $\mathfrak{S} a = \left\{\tau a\mid \tau\in\mathfrak{S},\, h(a) = t(\tau)\right\}$. If a tuple of matrices
$(B^{(1)},\ldots, B^{(n)})\in U(\mathfrak{S})$ is in
$\TrueIm{\widehat{\Phi}}$, we can recover all the maps $M_a,\ a\in
Q_1$ of its inverse image $(M,f)$ since for all $a\in Q_1$:
$B(\mathfrak{S}_{ta}a) = B(\mathfrak{S}_{ta})M_a$ and hence $M_a =
B(\mathfrak{S}_{ta})^{-1}B(\mathfrak{S}_{ta}a)$. Now observe that
$(B^{(1)},\ldots, B^{(n)})$ belongs to the image of $\widehat{\Phi}$
whenever, for all $\tau\in\Gamma(\alpha)$ and $a$ with $ta =
h(\tau)$, $B_{\tau a} = B_{\tau}M_a$. Using the expression received
for $a$, we rewrite this as
$$B_{\tau{a}} = B_{\tau}B(\mathfrak{S}_{ta})^{-1}B(\mathfrak{S}_{ta}a).\eqno{(3.1)}$$
Collected together, all these equations define a Zarisky
closed subvariety $X(\mathfrak{S})$ in $U(\mathfrak{S})$. Gluing
them we obtain a closed subvariety
$X_0(\alpha)\subseteq Z(\alpha)$ that coincides with
$\TrueIm\widehat{\Phi}$. Consequently, $\TrueIm\widehat{\Phi}$
is a locally closed subvariety in
$\IHom_{\alpha}\left(\widehat{J}\right)$.
\end{proof}

From the proof of this proposition we infer

\begin{corollary} The map $\widehat{\Phi}:\repq^s(Q,\alpha,\zeta)\rightarrow
\TrueIm\widehat{\Phi} = X_0(\alpha)$ is a $GL(\alpha)$-equivariant
isomorphism of algebraic varieties.
\end{corollary}

\begin{proof} To construct the inverse morphism we need to find ways of
recovering a pair $(M,f)$ possessing its image
$(B^{(1)},\ldots,B^{(n)})\in\IHom_{\alpha}(\widehat{J})$. But
$f_{iq}$ comes as $B_{f_{iq}}$ and $M_a =
B(\mathfrak{S}_{ta})^{-1}B(\mathfrak{S}_{ta}a)$ in $U(\mathfrak{S})$
while on intersections $U(\mathfrak{S})\cap U(\mathfrak{T})$ the
equality $B(\mathfrak{S}_{ta})^{-1}B(\mathfrak{S}_{ta}a) =
B(\mathfrak{T}_{ta})^{-1}B(\mathfrak{T}_{ta}a)$ is a direct
consequence of the equations defining $X_0$.

The equivariance of this isomorphism is obvious.
\end{proof}

From now on we will identify each $X(\mathfrak{S})\subseteq\IHom_{\alpha}(\widehat{J})$ with its preimage $\repq(Q,\mathfrak{S}) = \widehat{\Phi}^{-1}(X(\mathfrak{S}))\subseteq\repq^s(q,\alpha,\zeta)$ and treat it as a subset in $\repq^s(Q,\alpha,\zeta)$ whenever needed.

%



\pagebreak

\begin{theorem} In the notation given above
\begin{itemize}
\item[(1)] We have
$$\repq^s(Q,\alpha,\zeta) = \bigcup_{\mbox{$\mathfrak{S}$ is a $J$-skeleton}}X(\mathfrak{S}),$$
where $X(\mathfrak{S})$ are Zariski open subsets of $\repq^s(Q,\alpha,\zeta)$ such that $X(\mathfrak{S})\cong\GL(\alpha)\times\mathbb{A}^N$, for some positive integer $N$, and the restriction to $X(\mathfrak{S})$ of the quotient map is the projection onto the second factor. In particular, $$\mathcal{M}^s(Q,\alpha,\zeta) = \bigcup_{\mbox{$\mathfrak{S}$ is a $J$-skeleton}} X(\mathfrak{S})/\!\!/GL(\alpha)$$
is a covering by open subspaces isomorphic to affine spaces.
\item[(2)] The quotient space $\mathcal{M}^s(Q, \alpha, \zeta)$ is isomorphic to a locally closed subvariety in $\Grass_{\alpha}(\widehat{J})$.
\end{itemize}
\end{theorem}

\begin{proof} Again it will be convenient for us to view
$\IHom_{\alpha}(\widehat{J})$ as a Zariski open subset in
$\prod_{i\in Q_0}\Mat_{|\widetilde{\Gamma}_i(\alpha)|\times\alpha_i}(\Bbbk)$. Fix a
$J$-skeleton $\mathfrak{S}$ and assume that its elements are somehow
ordered (for example, lexicographically). For a tuple of matrices $B = (B^{(1)},\ldots, B^{(n)})$ define $B_{[\mathfrak{S}]}$ as the $n$-tuple of matrices with $B_{[\mathfrak{S}]}^{(i)}$
a submatrix of $B^{(i)}$ consisting of all the rows of all $B(\mathfrak{S}a)$, for $a\in Q_1$ with $t(a) = i$,
that do not occur in $B(\mathfrak{S}_i)$. Also denote by
$B_{\widehat{\mathfrak{S}}}$ a tuple of matrices with $B_{\widehat{\mathfrak{S}}}^{(i)}$ obtained as a
union of $B_{[\mathfrak{S}]}^{(i)}$ with all $B_{f_{ik_i}}$, for $f_{ik_i}\notin\mathfrak{S}$. Denote by $T_i$ the number of rows in $B_{[\mathfrak{S}]}^{(i)}$.

Consider the morphism $\pi_{\mathfrak{S}}: \prod_{i\in Q_0}\Mat_{|\widetilde{\Gamma}_i(\alpha)|\times\alpha_i}(\Bbbk)\rightarrow
\prod_{i\in Q_0}\Mat_{T_i\times\alpha_i}(\Bbbk)$ defined by $B^{(i)}\mapsto \left(B^{(i)}\cdot
B(\mathfrak{S}_i)^{-1}\right)_{\widehat{\mathfrak{S}}}$. We claim that
it provides the quotient morphism for the action of $GL(\alpha)$ on~$U(\mathfrak{S})$.

First of all, $\pi_{\mathfrak{S}}$ is $GL(\alpha)$-invariant, since, for
$g\in GL(\alpha)$, we have $\pi_{\mathfrak{S}}(g\cdot B)^{(i)} =
\pi_{\mathfrak{S}}(Bg^{-1})^{(i)} = \left(B^{(i)}g_i^{-1}\cdot
g_iB(\mathfrak{S}_i)^{-1}\right)_{\widehat{\mathfrak{S}}} =
\pi_{\mathfrak{S}}(B_i)$. We now prove that $\pi_{\mathfrak{S}}$ is
surjective. Let $C\in\prod_{i\in Q_0}\Mat_{T_i\times\alpha_i}(\Bbbk)$. Recall that each row of
each $C^{(i)}$ corresponds to a path from
$\bigcup_{a\in Q_1}\left(\mathfrak{S}a\backslash\mathfrak{S}\right)\cup\left\{f_i\notin\mathfrak{S}\right\}$. Take
an identity $\alpha_i\times\alpha_i$-matrix $E^{(i)}$ and put its $j$-th row into
correspondence with the $j$-th path from $\mathfrak{S}_i$ (with
respect to the order we introduced at the beginning of the proof).
Now, add to $C^{(i)}$ all the rows of $E^{(i)}$ corresponding to paths from
$\mathfrak{S}\cap(\bigcup_{a\in Q_1,\,ta = i}\mathfrak{S}a)$ and denote the matrix received by
$\widetilde{C}^{(i)}$. The first step will be now to recover a stable pair $(M^C,f^C)$, then we will use it to obtain a matrix in
$\pi_{\mathfrak{S}}^{-1}(C)$. Put $M^C_a = \widetilde{C}(\mathfrak{S}a)$, for all $a\in Q_1$,
and
$$f^C_i = \begin{cases} E_{f_i}^{(i)},\mbox{\ if $f_i\in\mathfrak{S}$},\\
C_{f_i}^{(i)},\mbox{\ otherwise}.
\end{cases}$$
Finally, set $B^C = \Phi_{(M^C,f^C)}$. One should show now that
$\pi_{\mathfrak{S}}(B^C) = C$. But $B(\mathfrak{S}_i) = E^{(i)}$, so
$B(\mathfrak{S}a) = B(\mathfrak{S})a = M_a =
\widetilde{C}(\mathfrak{S}a)$, for all $a\in Q_1$, implying that
$\pi_{\mathfrak{S}}(B)_{[\mathfrak{S}]} =
B_{[\mathfrak{S}]} = C_{[\mathfrak{S}]}$, for $a\in Q_1$.
Analogously, $B_{f_i}^{(i)} = C_{f_i}^{(i)}$, for all $f_i\notin\mathfrak{S}$.
Thus, $\pi_{\mathfrak{S}}(B) = C$ and the surjectivity is proven.

Now we should show that fibers of $\pi_{\mathfrak{S}}$ coincide with
$GL(\alpha)$-orbits in $U(\mathfrak{S})$. Observe, that for the above
constructed $B^C\in\pi_{\mathfrak{S}}^{-1}(C)$ we have $\left((B^C)^{(i)}\cdot
B^C(\mathfrak{S}_i)^{-1}\right)(\mathfrak{S}) = E^{(i)}$, an identity matrix.
But it is easy to see that a $GL(\alpha)$-orbit in $U(\mathfrak{S})$
contains only one tuple of matrices with this property. So, any $B\in\pi_{\mathfrak{S}}^{-1}(C)$ equals $g(B)\cdot B^C$, where $g(B)_i = B(\mathfrak{S}_i)$. 

An isomorphism $\GL(\alpha)\times\prod_{i\in Q_0}\Mat_{T_i\times\alpha_i}(\Bbbk)\rightarrow X(\mathfrak{S})$ is therefore obtained by sending a pair $\left(g,C\right)$ to $g\cdot B^C = \left((B^C)^{(i)}g^{-1}_i\right)_{i\in Q_0}$.


To prove the second part it is sufficient to check that each $X(\mathfrak{S})/\!\!/GL(\alpha)$ embeds into $\Grass(\mathfrak{S}) := U(\mathfrak{S})/\!\!/GL(\alpha)$ as a locally closed subvariety. Observe that the natural projection $\pi_0:U(\mathfrak{S})\twoheadrightarrow \Grass(\mathfrak{S})$ may be viewed as the map $\prod_{i\in Q_0}\Mat_{|\widetilde{\Gamma}_i(\alpha)|\times\alpha_i}(\Bbbk)\rightarrow\prod_{i\in Q_0}\Mat_{(|\widetilde{\Gamma}_i(\alpha)| - \alpha_i)\times\alpha_i}(\Bbbk)$, $B\mapsto\left(B^{(i)}\cdot
B(\mathfrak{S}_i)^{-1}\right)_{\Gamma(\alpha)\backslash\mathfrak{S}}$. Hence $X(\mathfrak{S})/\!\!/GL(\alpha)$ is isomorphic to a subvariety in $\Grass(\mathfrak{S})$ defined by equations
$$C_{\tau{a}} = C_{\tau}C(\mathfrak{S}_{ta}a),$$
for all paths $\tau$ and arrows $a$ with $\tau{a}\notin(\bigcup_{a\in Q_1,\,ta = i}\mathfrak{S}a)$ .
\end{proof}

\begin{corollary} In the settings of Theorem 1 we have the following
\begin{itemize}
\item[(1)] Each stable pair $(M,f)$ possesses a finite family of normal forms, each normal form corresponding to a $J$-skeleton of $(M,f)$. If $\mathfrak{S}$ is a $J$-skeleton of $(M,f)$, then the respective normal form of $(M,f)$ is $\left(M^C,f^C\right)$, where $C = \pi_{\mathfrak{S}}(M,f)$.
\item[(2)] The following procedure may be used to determine whether two stable pair $(M,f)$ and $(M',f')$ are isomorphic.
    \begin{itemize}
\item[(a)] Compute tuples of matrices $B$ and $B'$ corresponding to $(M,f)$ and $(M',f')$.
\item[(b)] Find their $J$-skeleta by seeking non-degenerate maximal minors of $B$ and $B'$.
\item[(c)] If they have no common $J$-skeleta, the pairs are not isomorphic.
\item[(d)] If $\mathfrak{S}$ a common $J$-skeleton, compute $\left(M^{C},f^{C}\right)$ and $\left(M^{C'},f^{C'}\right)$ for $C = \pi_{\mathfrak{S}}(B)$ and $C' = \pi_{\mathfrak{S}}(B')$.
\item[(e)] The pairs $(M,f)$ and $(M',f')$ are isomorphic if and only if $\left(M^{C},f^{C}\right) = \left(M^{C'},f^{C'}\right)$.
\end{itemize}
\end{itemize}
\end{corollary}

\begin{remark} {\rm Dimensions of the affine spaces covering the quotient space equal the dimension of the quotient itself, i.e., the difference $\dim\repq(Q,\alpha,\zeta) - \dim\GL(\alpha)$. For $Q = L_{q,k}$ we have $\dim{X(\mathfrak{S})/\!\!/\GL_m} = \dim{\repq(L_q,m,k)} - \dim\GL_m = (m^2q + mk) - m^2 = m(mq + k - m)$, so that $X(\mathfrak{S})/\!\!/GL_m\cong\mathbb{A}^{m(mq + k - m)}$.}
\end{remark}

We have described affine charts covering the quotient. Since all of them are affine spaces, on each we obtain a convenient system of local coordinates. The transition functions between these coordinates may be described in the following way. Let $\mathfrak{S}$ and $\mathfrak{T}$ be two $J$-skeleta, and $B\in\prod_{i\in Q_0}\Mat_{|\widetilde{\Gamma}_i(\alpha)|\times\alpha_i}(\Bbbk)$ be a matrix representing a point in $\repq^s(Q,\alpha,\zeta)$. Then we may express $B$ in matrix elements of $(B^{(i)}\cdot B(\mathfrak{S}_i)^{-1})_{\widehat{\mathfrak{S}}}$, for $i\in Q_0$, and further obtain the expression for $(B^{(j)}\cdot B(\mathfrak{T}_j)^{-1})_{\widehat{\mathfrak{T}}}$.

The same procedure may be used to establish relations between normal forms of a stable pair constructed using different $J$-skeleta.

%
%


\section{Examples}

Let $Q$ be the quiver $L_q$ and $\alpha = (m)$. It is easy to see that every $J$-skeleton that may occur in this case is a subset of $\left\{f_iW(a_1,\ldots,a_q)\mid\mbox{$W$ is a word in $a_j$ of length at most $m-1$}\right\}$, so we set $$\widetilde{\Gamma}(m) = \left\{f_iW(a_1,\ldots,a_q)\mid\mbox{$W$ is a word in $a_j$ of length at most $m$}\right\},$$ $\widehat{J} = \bigoplus_{i,W, length(W) \leqslant m}\Bbbk u^{i,W}$ and $\varphi_{i,W}(m) = f_iW(a_1,\ldots,a_q)m$.


\begin{example} {\rm Let $q = k = 1$. This encodes the problem of classifying pairs $(a,f)$, where $a$ is a linear operator on $\Bbbk^m$ and $f$ is a linear function. The extended quiver $L_1^{(1)} = L_{1,1}$ is $$\xymatrix{1\ar@(ul,dl)[]_a\ar[r]^f & {\infty}}$$

There is only one $J$-skeleton, namely $\mathfrak{S} = \{f, fa,\ldots,fa^{m-1}\}$. Therefore, $\widetilde{\Gamma}(m) = \{f, fa,\ldots,fa^{m}\}$ and $\widehat{J} = \bigoplus_{i = 0}^mV_1^{(fa^i)}$, where $V_1 = \Bbbk$, so that $\widehat{J} = \Bbbk^{m+1}$ as a vector space. Furthermore, we have
$$\widehat{\Phi}:\repq(L_1,m,1)\rightarrow\Mat_{(m+1)\times m}(\Bbbk),\quad (a,f)\mapsto\widehat{\Phi}_{(a,f)}=\begin{pmatrix}f\\fa\\ \vdots\\ fa^m\end{pmatrix}.$$
One can easily see that the stability condition (that is injectivity of $\widehat{\Phi}_{(a,f)}$) is in this case equivalent to $f$ being a cyclic vector for the natural action $GL_m:(\Bbbk^m)^*$.
The subset $X(\mathfrak{S})\subseteq\Mat_{(m+1)\times m}(\Bbbk)$ coincides with $U(\mathfrak{S})$, i.e., consists of matrices $B = (b_{ij})$ with first $m$ rows linear independent. Indeed, the only condition $B_{fa^m} = B_{fa^{m-1}}B(\mathfrak{S})^{-1}B(\mathfrak{S}a)$ holds identically. Next, for a $(m+1)\times m$-matrix $A$ we have $A_{\widehat{\mathfrak{S}}} = A_{fa^{m}}$, hence the quotient map $\pi_{\mathfrak{S}}$ sends
$$\pi_{\mathfrak{S}}: B = (b_{ij})\mapsto \left(b_{m+1,1},\ldots,b_{m+1,m}\right)\begin{pmatrix}
b_{11} & \ldots & b_{1m}\\
\vdots & & \vdots\\
b_{m1} & \ldots & b_{mm}\end{pmatrix}^{-1}.$$

The quotient is isomorphic to $\mathbb{A}^m$ and $\pi_{\mathfrak{S}}$ admits a section $\mathbb{A}^m\rightarrow\repq(L_1,m,1)$ sending $C = (c_0,\ldots,c_{m-1})$ to the pair
$$f^C = (1,0,\ldots,0),\qquad a^C = \begin{pmatrix}
0 & 1 & 0 & \ldots & 0\\
\vdots & \ddots & \ddots & \ddots & \vdots\\
0&  & 0 & 1 & 0\\
0 & \hdotsfor{2} & 0 & 1\\
c_0 & c_1 & \ldots & c_{m-2} & c_{m-1}
\end{pmatrix},$$
where $c_i$ are the coefficients of the characteristic polynomial of $a$.}
\end{example}

\begin{example} {\rm Let $q = 1$, $m = k = 2$. The corresponding extended quiver $L_1^{(2)} = L_{1,2}$ is
$$\xymatrix{1\ar@(ul,dl)[]_a\ar@/^0.5pc/[r]^{f_1}\ar@/_0.5pc/[r]_{f_2} & {\infty}}$$
There are three possible $J$-skeleta: $\mathfrak{S}_1 = \{f_1,f_1a\}$, $\mathfrak{S}_2 = \{f_2, f_2a\}$, $\mathfrak{S}_3 = \{f_1,f_2\}$. Hence, $\widetilde{\Gamma}(m) = \left\{f_1,f_2,f_1a,f_2a,f_1a^2,f_2a^2\right\}$, so that
$$\widehat{\Phi}:\repq(L_1,2,2)\rightarrow\Mat_{6\times2}(\Bbbk),\quad (a,f_1,f_2)\mapsto\widehat{\Phi}_{(a,f_1,f_2)}=\begin{pmatrix}f_1\\f_2\\f_1a\\f_2a\\f_1a^2\\f_2a^2\end{pmatrix}.$$
In particular, $\widehat{J} = \Bbbk^6$.

First of all we describe the subset $\TrueIm{\widehat{\Phi}} = X(\mathfrak{S}_1)\cup X(\mathfrak{S}_2)\cup X(\mathfrak{S}_3)\subseteq\Mat_{6\times2}(\Bbbk)$. The chart $X(\mathfrak{S}_1)$ consists of $6\times2$-matrices $B = (b_{ij})$ with $|B(\mathfrak{S}_1)| = \left|\begin{smallmatrix}b_{11} & b_{12}\\ b_{31} & b_{32}\end{smallmatrix}\right|\ne 0$ satisfying conditions $B_{f_2a} = B_{f_2}B(\mathfrak{S}_1)^{-1}B(\mathfrak{S}_1a)$ and $B_{f_2a^2} = B_{f_2a}B(\mathfrak{S}_1)^{-1}B(\mathfrak{S}_1a)$, that may be expanded as
$$(b_{41}, b_{42}) = (b_{21}, b_{22})\begin{pmatrix}b_{11} & b_{12}\\ b_{31} & b_{32}\end{pmatrix}^{-1}\begin{pmatrix}b_{31} & b_{32}\\ b_{51} & b_{52}\end{pmatrix}$$
and $$(b_{61}, b_{62}) = (b_{41}, b_{42})\begin{pmatrix}b_{11} & b_{12}\\ b_{31} & b_{32}\end{pmatrix}^{-1}\begin{pmatrix}b_{31} & b_{32}\\ b_{51} & b_{52}\end{pmatrix}$$
respectively. Similarly we find that $X(\mathfrak{S}_2)\subseteq\Mat_{{6\times2}(\Bbbk)}$ consists of matrices $B = (b_{ij})$ with $|B(\mathfrak{S}_2)| = \left|\begin{smallmatrix}b_{21} & b_{22}\\ b_{41} & b_{42}\end{smallmatrix}\right|\ne 0$ and
$$\begin{pmatrix}b_{31} & b_{32}\\ b_{51} & b_{52}\end{pmatrix} = \begin{pmatrix}b_{11} & b_{12}\\b_{31} & b_{32}\end{pmatrix}\begin{pmatrix}b_{21} & b_{22}\\ b_{41} & b_{42}\end{pmatrix}^{-1}\begin{pmatrix}b_{41} & b_{42}\\ b_{61} & b_{62}\end{pmatrix}.$$
Finally, $X(\mathfrak{S}_3)\subseteq\Mat_{6\times2}(\Bbbk)$ is a subset consisting of matrices $B = (b_{ij})$ satisfying $B(\mathfrak{S}_3) = \left|\begin{smallmatrix}b_{11} & b_{12}\\ b_{21} & b_{22}\end{smallmatrix}\right|\ne 0$ and
$$\begin{pmatrix}b_{51} & b_{52}\\ b_{61} & b_{62}\end{pmatrix} = \begin{pmatrix}b_{31} & b_{32}\\b_{41} & b_{42}\end{pmatrix}\begin{pmatrix}b_{11} & b_{12}\\ b_{21} & b_{22}\end{pmatrix}^{-1}\begin{pmatrix}b_{31} & b_{32}\\ b_{41} & b_{42}\end{pmatrix}.$$
Next, we write down the quotient maps $\pi_{\mathfrak{S}_i}:\repq(L_1,\mathfrak{S}_i)\cong X(\mathfrak{S}_i)\rightarrow\Mat_{2\times2}(\Bbbk)$;
$$\pi_{\mathfrak{S}_1}: (a,f_1,f_2)\mapsto \begin{pmatrix}
f_2\\
f_1a^2
\end{pmatrix}\begin{pmatrix}
f_1\\
f_1a
\end{pmatrix}^{-1},\ B = (b_{ij})\mapsto\begin{pmatrix}
b_{31} & b_{32}\\
b_{51} & b_{52}
\end{pmatrix}\begin{pmatrix}
b_{11} & b_{12}\\
b_{31} & b_{32}
\end{pmatrix}^{-1};$$
$$\pi_{\mathfrak{S}_2}: (a,f_1,f_2)\mapsto \begin{pmatrix}
f_1\\
f_2a^2
\end{pmatrix}\begin{pmatrix}
f_2\\
f_2a
\end{pmatrix}^{-1},\ B = (b_{ij})\mapsto\begin{pmatrix}
b_{41} & b_{42}\\
b_{61} & b_{62}
\end{pmatrix}\begin{pmatrix}
b_{21} & b_{22}\\
b_{41} & b_{42}
\end{pmatrix}^{-1};$$
$$\pi_{\mathfrak{S}_3}: (a,f_1,f_2)\mapsto \begin{pmatrix}
f_1a\\
f_2a
\end{pmatrix}\begin{pmatrix}
f_1\\
f_2
\end{pmatrix}^{-1},\ B = (b_{ij})\mapsto\begin{pmatrix}
b_{31} & b_{32}\\
b_{41} & b_{42}
\end{pmatrix}\begin{pmatrix}
b_{11} & b_{12}\\
b_{21} & b_{22}
\end{pmatrix}^{-1}.$$
Each $\pi_{\mathfrak{S}_i}$ admits a section $\Mat_{2\times2}(\Bbbk)\rightarrow\repq(L_1,\mathfrak{S}_i)$ sending $C = (x^{(i)}_{ab})$ to the triple $(a^C,f_1^C,f_2^C)$ with
$$
a^C = \begin{pmatrix}
0 & 1\\
x^{(1)}_{21} & x^{(1)}_{22}
\end{pmatrix},\
f^C_1 =(1, 0),\ f_2^C = (x^{(1)}_{11}, x^{(1)}_{12}),\mbox{ for $i = 1$}\eqno{(4.1)}$$
$$
a^C = \begin{pmatrix}
0 & 1\\
x^{(2)}_{21} & x^{(2)}_{22}
\end{pmatrix},\
f^C_1 =(x^{(2)}_{11}, x^{(2)}_{12}),\ f_2^C = (1,0),\mbox{ for $i = 2$}\eqno{(4.2)}$$
$$a^C = C = \begin{pmatrix}
x^{(1)}_{11} & x^{(1)}_{12}\\
x^{(1)}_{21} & x^{(1)}_{22}
\end{pmatrix},\ f^C_1 = (1, 0),\ f^C_2 = (0,1),\mbox{ for $i = 3$}.\eqno{(4.3)}$$

Now, if a triple $T:=(a,f_1,f_2)$ belongs to $\repq(Q,\mathfrak{S}_i)\cap\repq(Q,\mathfrak{S}_j)$, for some $i\ne j$, we obtain two normal forms $(a^{\pi_{\mathfrak{S}_i}(T)}, f_1^{\pi_{\mathfrak{S}_i}(T)}, f_2^{\pi_{\mathfrak{S}_i}(T)})$ and $(a^{\pi_{\mathfrak{S}_j}(T)}, f_1^{\pi_{\mathfrak{S}_j}(T)}, f_2^{\pi_{\mathfrak{S}_j}(T)})$ corresponding to $2\times2$-matrices $\pi_{\mathfrak{S}_i}(T) = (x_{ab}^{(i)})$ and $\pi_{\mathfrak{S}_j}(T) = (x_{ab}^{(j)})$. One easily computes that
$$\begin{pmatrix}
x^{(1)}_{11} & x^{(1)}_{12}\\
x^{(1)}_{21} & x^{(1)}_{22}
\end{pmatrix}  =
\begin{pmatrix}
f_2\\
f_1a^2
\end{pmatrix}\begin{pmatrix}
f_1\\
f_1a
\end{pmatrix}^{-1} =
\begin{pmatrix}
-\frac{\displaystyle x^{(3)}_{11}}{\displaystyle x^{(3)}_{12}} &
\frac{\displaystyle 1}{\displaystyle x^{(3)}_{12}}\\
x^{(3)}_{12}x^{(3)}_{21} - x^{(3)}_{11}x^{(3)}_{22} & x^{(3)}_{11} + x^{(3)}_{22}
\end{pmatrix},$$
$$\begin{pmatrix}
x^{(3)}_{11} & x^{(3)}_{12}\\
x^{(3)}_{21} & x^{(3)}_{22}
\end{pmatrix} = \begin{pmatrix}
-\frac{\displaystyle x^{(1)}_{11}}{\displaystyle x^{(1)}_{12}} &
\frac{\displaystyle 1}{\displaystyle x^{(1)}_{12}}\\
\frac{\displaystyle (x^{(1)}_{12})^2x^{(1)}_{21} - (x^{(1)}_{11})^2 - x^{(1)}_{11}x^{(1)}_{12}x^{(1)}_{22}}{\displaystyle x^{(1)}_{12}} & \frac{\displaystyle x^{(1)}_{11} + x^{(1)}_{12}x^{(1)}_{22}}{\displaystyle x^{(1)}_{12}}
\end{pmatrix},$$
$$\begin{pmatrix}
x^{(2)}_{11} & x^{(2)}_{12}\\
x^{(2)}_{21} & x^{(2)}_{22}
\end{pmatrix} = \begin{pmatrix}
-\frac{\displaystyle x^{(1)}_{11} + x^{(1)}_{12}x^{(1)}_{22}}{\displaystyle (x^{(1)}_{12})^2x^{(1)}_{21} - (x^{(1)}_{11})^2 - x^{(1)}_{11}x^{(1)}_{12}x^{(1)}_{22}} & \frac{\displaystyle x^{(1)}_{12}}{\displaystyle (x^{(1)}_{12})^2x^{(1)}_{21} - (x^{(1)}_{11})^2 - x^{(1)}_{11}x^{(1)}_{12}x^{(1)}_{22}}\\
x^{(1)}_{21} & x^{(1)}_{22}
\end{pmatrix}.$$

\bigskip

We return for a while to the case of arbitrary $m$, $q$ and $k$. Since the quotient is embedded into $\Grass_{m}(\widehat{J})$, we need to introduce a connection between local coordinates on $Y(\mathfrak{S}) := X(\mathfrak{S})/\!\!/GL(\alpha)$ and Pl\"ucker coordinates on $\Grass_m(\widehat{J})$. Observe that the latter are of the form $p_{\mathfrak{R}}$, for all subsets $\mathfrak{R}\subseteq\widetilde{\Gamma}(m)$ of cardinality $m$. Indeed, the natural projection $\Mat_{k(m+1)^q\times m}(\Bbbk)\supseteq\IHom_m(\widehat{J})\twoheadrightarrow\Grass_m(\widehat{J})$ maps a matrix $B$ to a point $\omega_B$, whose Pl\"ucker coordinates are $m\times m$-minors of $B$. So, we denote by $p_{\mathfrak{R}}$ the coordinate with corresponding minor consisting of rows prescribed by $\mathfrak{R}$. As for the local coordinates on $Y(\mathfrak{S})$, their definition together with Cramer's rule imply that they are indexed by $m$-element subsets $\mathfrak{R}$ in $\widetilde{\Gamma}(m)$ that may be obtained from $\mathfrak{S}$ by replacement of one of its elements by an element of $(\mathfrak{S}a\backslash\mathfrak{S})\cup\{f_i\notin\mathfrak{S}\}$. 

In Example 4.2, we have $\Grass_{m}(\widehat{J}) = \Grass_2(\Bbbk^6)$. As it was suggested before, we index Pl\"ucker coordinates by pairs of paths. For instance, $p_{f_2,f_{1}a^2}$ stands for $p_{25}$. For a point $\omega_{B}$ corresponding to a rank $m$ matrix $B\in\Mat_{6\times2}(\Bbbk)$ it comes as $\left|\begin{smallmatrix}b_{21} & b_{22}\\b_{51} & b_{52}\end{smallmatrix}\right|$, and this equals~$\left|\begin{smallmatrix}f_2\\f_1a^2\end{smallmatrix}\right|$ if $(b_{ij}) = \widehat{\Phi}((a,f_1,f_2))$. We obtain the expressions
$$\begin{pmatrix}
x^{(1)}_{11} & x^{(1)}_{12}\\
x^{(1)}_{21} & x^{(1)}_{22}
\end{pmatrix} = \begin{pmatrix}
\frac{\displaystyle p_{f_2,f_1a}}{\displaystyle p_{f_1,f_1a}} & -\frac{\displaystyle p_{f_1,f_2}}{\displaystyle p_{f_1,f_1a}}\\
-\frac{\displaystyle p_{f_1a,f_1a^2}}{\displaystyle p_{f_1,f_1a}} & \frac{\displaystyle p_{f_1,f_1a^2}}{\displaystyle p_{f_1,f_1a}}
\end{pmatrix},\quad \begin{pmatrix}
x^{(2)}_{11} & x^{(2)}_{12}\\
x^{(2)}_{21} & x^{(2)}_{22}
\end{pmatrix} = \begin{pmatrix}
\frac{\displaystyle p_{f_1,f_2a}}{\displaystyle p_{f_2,f_2a}} & -\frac{\displaystyle p_{f_1,f_2}}{\displaystyle p_{f_2,f_2a}}\\
-\frac{\displaystyle p_{f_2a,f_2a^2}}{\displaystyle p_{f_2,f_2a}} & \frac{\displaystyle p_{f_2,f_2a^2}}{\displaystyle p_{f_2,f_2a}}
\end{pmatrix},$$
$$\begin{pmatrix}
x^{(3)}_{11} & x^{(3)}_{12}\\
x^{(3)}_{21} & x^{(3)}_{22}
\end{pmatrix} = \begin{pmatrix}
-\frac{\displaystyle p_{f_2,f_1a}}{\displaystyle p_{f_1,f_2}} & \frac{\displaystyle p_{f_1,f_1a}}{\displaystyle p_{f_1,f_2}}\\
-\frac{\displaystyle p_{f_2,f_2a}}{\displaystyle p_{f_1,f_2}} & \frac{\displaystyle p_{f_1,f_2a}}{\displaystyle p_{f_1,f_2}}
\end{pmatrix}.$$
Now, we determine the equations in Pl\"ucker coordinates that define the closure of the quotient in $\Grass_2(\widehat{J})$ in Example 4.2. First of all, there are Pl\"ucker relations. We also have the following equations coming from transition relations between $x^{(i)}_{ab}$ and $x^{(j)}_{cd}$
$$p_{f_1a,f_1a^2}p_{f_2,f_2a} = p_{f_2a,f_2a^2}p_{f_1,f_1a},\quad p_{f_1,f_1a^2}p_{f_2,f_2a} = p_{f_2,f_2a^2}p_{f_1,f_1a},$$
$$p_{f_1a,f_1a^2}p_{f_1,f_2}^2 = p_{f_1,f_1a}(p_{f_1,f_1a}p_{f_2,f_2a} - p_{f_2,f_1a}p_{f_1,f_2a}),$$
$$p_{f_2a,f_2a^2}p_{f_1,f_2}^2 = p_{f_2,f_2a}(p_{f_1,f_1a}p_{f_2,f_2a} - p_{f_2,f_1a}p_{f_1,f_2a}),$$
$$p_{f_1,f_1a^2}p_{f_1,f_2} = p_{f_1,f_1a}(p_{f_1,f_2a} - p_{f_2,f_1a}),$$
$$p_{f_2,f_2a^2}p_{f_1,f_2} = p_{f_2,f_2a}(p_{f_1,f_2a} - p_{f_2,f_1a}).$$
As we see, not all $p_{\mathfrak{R}}$, $\mathfrak{R}\subseteq\widetilde{\Gamma}(m)$, occur as numerators of local coordinates on the quotient. Those that do not occur we will call {\it exceed} and others {\it essential}. We claim that exceed homogeneous coordinates may be eliminated, i.e., that we are able to express them as polynomials in local coordinates in each affine chart. Indeed, we may express $f_i$ and $a_j$, and then compute all maximal minors of the matrix of $\Phi_{(a,f)}$. For instance, we have
$$\frac{p_{f_1a,f_2a}}{p_{f_1,f_1a}} = \frac{p_{f_1,f_2}p_{f_1a,f_1a^2}}{p_{f_1,f_1a^2}^2},
\qquad\frac{p_{f_1a,f_2a}}{p_{f_2,f_2a}} = \frac{p_{f_1,f_2}p_{f_2a,f_2a^2}}{p_{f_2,f_2a^2}^2},$$
$$\frac{p_{f_1a,f_2a}}{p_{f_1,f_2}} = \frac{p_{f_1,f_1a}p_{f_2,f_2a} - p_{f_1,f_2a}p_{f_1a,f_2}}{p_{f_1,f_2}^2},$$
$$\frac{p_{f_1,f_2a^2}}{p_{f_1,f_1a}} = \frac{p_{f_1,f_2}p_{f_1a,f_1a^2}p_{f_1,f_1a} + p_{f_2,f_1a}p_{f_1,f_1a^2}p_{f_1,f_1a} - p_{f_1,f_2}p_{f_1,f_1a^2}^2}{p_{f_1,f_1a}^3},$$
$$\frac{p_{f_1,f_2a^2}}{p_{f_1,f_2}} = \frac{p_{f_1,f_2a}^2 - p_{f_1,f_1a}p_{f_2,f_2a}}{p_{f_1,f_2}^2}.$$

Returning back to the general case, we may consider an obvious projection-like map $\mathcal{M}^s(L_q,m,k)\rightarrow\mathbb{P}^N$, where $N+1$ is the number of non-exceed coordinates. Clearly it is well defined. Denote by $\widehat{Y}_0$ the image of the quotient in $\mathbb{P}^N$. We claim now that the closure $\widehat{Y}$ of $\widehat{Y}_0$ (or at least something containing $\widehat{Y}$ as an irreducible component) is defined in $\mathbb{P}^N$ solely by the equations that come from transition relations on the quotient, i.e., that Pl\"ucker equations are no more required. The reason is that on each affine chart $Y(\mathfrak{S})$, where $\mathfrak{S}$ is a $J$-skeleton, we can express all non-exceed coordinates $p_{\mathfrak{R}}$ in local coordinates of $Y(\mathfrak{S})$ using only this kind of equations, and since $Y(\mathfrak{S})$ is an affine space, these local coordinates are algebraically independent. So, no additional equations are needed.

In the case $q = 1$, $m = k = 2$ we have $15$ Pl\"ucker coordinates, only $9$ of them are essential and other $6$ are exceed. So, the quotient may be embedded as a locally closed subset into $\mathbb{P}^8$.
}
\end{example}

\begin{example} {\rm Let $q = m = 2$, $k = 1$. Denote the loops by $a$ and $b$. There are then two possible $J$-skeleta: $\mathfrak{S}_1 = \{f,fa\}$ and $\mathfrak{S}_2 = \{f,fb\}$. Hence $\widetilde{\Gamma}(m) = \left\{f,fa,fa^2,fb,fab,fba,fb^2\right\}$ (we fix this order of paths and construct the map $\widehat{\Phi}$ according to it) and $\widehat{J} = \Bbbk^7$.

It is not hard to see that
$$X(\mathfrak{S}_1) = \left\{B\in\Mat_{7\times2}(\Bbbk)\left|\begin{matrix}\begin{vmatrix}b_{11} & b_{12}\\b_{21} & b_{22}\end{vmatrix}\ne 0,\\
\begin{pmatrix}b_{61} & b_{62}\end{pmatrix} = \begin{pmatrix}b_{31} & b_{32}\end{pmatrix}\begin{pmatrix}b_{11} & b_{12}\\ b_{21} & b_{22}\end{pmatrix}^{-1}\begin{pmatrix}b_{21} & b_{22}\\ b_{41} & b_{42}\end{pmatrix}\\
\begin{pmatrix}b_{71} & b_{72}\end{pmatrix} = \begin{pmatrix}b_{31} & b_{32}\end{pmatrix}\begin{pmatrix}b_{11} & b_{12}\\ b_{21} & b_{22}\end{pmatrix}^{-1}\begin{pmatrix}b_{31} & b_{32}\\ b_{51} & b_{52}\end{pmatrix}\end{matrix}\right.\right\}$$
and
$$X(\mathfrak{S}_2) = \left\{B\in\Mat_{7\times2}(\Bbbk)\left|\begin{matrix}\begin{vmatrix}b_{11} & b_{12}\\b_{31} & b_{32}\end{vmatrix}\ne 0,\\
\begin{pmatrix}b_{41} & b_{42}\end{pmatrix} = \begin{pmatrix}b_{21} & b_{22}\end{pmatrix}\begin{pmatrix}b_{11} & b_{12}\\ b_{31} & b_{32}\end{pmatrix}^{-1}\begin{pmatrix}b_{21} & b_{22}\\ b_{61} & b_{62}\end{pmatrix}\\
\begin{pmatrix}b_{51} & b_{52}\end{pmatrix} = \begin{pmatrix}b_{21} & b_{22}\end{pmatrix}\begin{pmatrix}b_{11} & b_{12}\\ b_{31} & b_{32}\end{pmatrix}^{-1}\begin{pmatrix}b_{31} & b_{32}\\ b_{71} & b_{72}\end{pmatrix}\end{matrix}\right.\right\}$$
The quotient maps are
$$\pi_{\mathfrak{S}_1}: (a,b,f)\mapsto \begin{pmatrix}
fa^2\\
fb\\
fab
\end{pmatrix}\begin{pmatrix}
f\\
fa
\end{pmatrix}^{-1},\ B = (b_{ij})\mapsto\begin{pmatrix}
b_{31} & b_{32}\\
b_{41} & b_{42}\\
b_{51} & b_{52}
\end{pmatrix}\begin{pmatrix}
b_{11} & b_{12}\\
b_{21} & b_{22}
\end{pmatrix}^{-1},$$
$$\pi_{\mathfrak{S}_2}: (a,b,f)\mapsto \begin{pmatrix}
fa\\
fba\\
fb^2
\end{pmatrix}\begin{pmatrix}
f\\
fb
\end{pmatrix}^{-1},\ B = (b_{ij})\mapsto\begin{pmatrix}
b_{21} & b_{22}\\
b_{61} & b_{62}\\
b_{71} & b_{72}
\end{pmatrix}\begin{pmatrix}
b_{11} & b_{12}\\
b_{31} & b_{32}
\end{pmatrix}^{-1},$$
so that $X(\mathfrak{S}_1)/\!\!/\GL(\alpha)\cong X(\mathfrak{S}_1)/\!\!/\GL(\alpha)\cong\Mat_{3\times2}(\Bbbk)\cong\mathbb{A}^6$. The sections $s_i: \Mat_{3\times2}(\Bbbk)\rightarrow X(\mathfrak{S}_i)\xrightarrow{\sim}\repq(Q,\mathfrak{S}_i)$ are as follows:
$$\mathfrak{s}_1: \begin{pmatrix}
x^{(1)}_{11} & x^{(1)}_{12}\\
x^{(1)}_{21} & x^{(1)}_{22}\\
x^{(1)}_{31} & x^{(1)}_{32}
\end{pmatrix}\mapsto
\left(\begin{pmatrix}
0 & 1\\
x^{(1)}_{11} & x^{(1)}_{12}
\end{pmatrix},
\begin{pmatrix}
x^{(1)}_{21} & x^{(1)}_{22}\\
x^{(1)}_{31} & x^{(1)}_{32}
\end{pmatrix},
\begin{pmatrix}
1 & 0\end{pmatrix}\right)
,$$
$$\mathfrak{s}_2: \begin{pmatrix}
x^{(2)}_{11} & x^{(2)}_{12}\\
x^{(2)}_{21} & x^{(2)}_{22}\\
x^{(2)}_{31} & x^{(2)}_{32}
\end{pmatrix}\mapsto
\left(\begin{pmatrix}
x^{(2)}_{11} & x^{(2)}_{12}\\
x^{(2)}_{21} & x^{(2)}_{22}
\end{pmatrix},
\begin{pmatrix}
0 & 1\\
x^{(2)}_{31} & x^{(2)}_{32}
\end{pmatrix},
\begin{pmatrix}
1 & 0\end{pmatrix}\right)
,$$
with transition functions
$$
\begin{pmatrix}
x^{(1)}_{11} & x^{(1)}_{12}\\
x^{(1)}_{21} & x^{(1)}_{22}\\
x^{(1)}_{31} & x^{(1)}_{32}
\end{pmatrix} =
\begin{pmatrix}
x^{(2)}_{12}x^{(2)}_{21} - x^{(2)}_{11}x^{(2)}_{22} & x^{(2)}_{11} + x^{(2)}_{22}\\
-\frac{\displaystyle x^{(2)}_{11}}{\displaystyle x^{(2)}_{12}} & \frac{\displaystyle 1}{\displaystyle x^{(2)}_{12}}\\
\frac{\displaystyle (x^{(2)}_{12})^2x^{(2)}_{31} - (x^{(2)}_{11})^2 - x^{(2)}_{11}x^{(2)}_{12}x^{(2)}_{32}}{\displaystyle x^{(2)}_{12}} & \frac{\displaystyle x^{(2)}_{11} + x^{(2)}_{12}x^{(2)}_{32}}{\displaystyle x^{(2)}_{12}}
\end{pmatrix}$$
and
$$
\begin{pmatrix}
x^{(2)}_{11} & x^{(2)}_{12}\\
x^{(2)}_{21} & x^{(2)}_{22}\\
x^{(2)}_{31} & x^{(2)}_{32}
\end{pmatrix} =
\begin{pmatrix}
-\frac{\displaystyle x^{(1)}_{21}}{\displaystyle x^{(2)}_{22}} & \frac{\displaystyle 1}{\displaystyle x^{(1)}_{22}}\\
\frac{\displaystyle (x^{(1)}_{22})^2x^{(1)}_{11} - (x^{(1)}_{21})^2 - x^{(1)}_{12}x^{(1)}_{22}x^{(1)}_{21}}{\displaystyle x^{(1)}_{22}} & \frac{\displaystyle x^{(1)}_{21} + x^{(1)}_{12}x^{(1)}_{22}}{\displaystyle x^{(1)}_{22}}\\
x^{(1)}_{22}x^{(1)}_{31} - x^{(1)}_{21}x^{(1)}_{32} & x^{(1)}_{21} + x^{(1)}_{32}
\end{pmatrix}
$$
Having
$$\begin{pmatrix}
x^{(1)}_{11} & x^{(1)}_{12}\\
x^{(1)}_{21} & x^{(1)}_{22}\\
x^{(1)}_{31} & x^{(1)}_{32}
\end{pmatrix} =
\begin{pmatrix}
-\frac{\displaystyle p_{fa,fa^2}}{\displaystyle p_{f,fa}} & \frac{\displaystyle p_{f,fa^2}}{\displaystyle p_{f,fa}}\\
-\frac{\displaystyle p_{fa,fb}}{\displaystyle p_{f,fa}} & \frac{\displaystyle p_{f,fb}}{\displaystyle p_{f,fa}}\\
-\frac{\displaystyle p_{fa,fab}}{\displaystyle p_{f,fa}} & \frac{\displaystyle p_{f,fab}}{\displaystyle p_{f,fa}}
\end{pmatrix},\quad
\begin{pmatrix}
x^{(2)}_{11} & x^{(2)}_{12}\\
x^{(2)}_{21} & x^{(2)}_{22}\\
x^{(2)}_{31} & x^{(2)}_{32}
\end{pmatrix} =
\begin{pmatrix}
\frac{\displaystyle p_{fa,fb}}{\displaystyle p_{f,fb}} & \frac{\displaystyle p_{f,fa}}{\displaystyle p_{f,fb}}\\
-\frac{\displaystyle p_{fb,fba}}{\displaystyle p_{f,fb}} & \frac{\displaystyle p_{f,fba}}{\displaystyle p_{f,fb}}\\
-\frac{\displaystyle p_{fb,fb^2}}{\displaystyle p_{f,fb}} & \frac{\displaystyle p_{f,fb^2}}{\displaystyle p_{f,fb}}
\end{pmatrix},$$
we obtain the following set of equations
$$p_{fa,fa^2}p_{f,fb}^2 + p_{f,fa}(p_{fa,fb}p_{f,fba} + p_{f,fa}p_{fb,fba}) = 0,$$
$$p_{fb,fb^2}p_{f,fa}^2 - p_{f,fa}(p_{f,fb}p_{fa,fab} - p_{fa,fb}p_{f,fab}) = 0,$$
$$p_{f,fa^2}p_{f,fb} - p_{f,fa}(p_{fa,fb} + p_{f,fba}) = 0,$$
$$p_{f,fb^2}p_{f,fa} - p_{f,fb}(p_{f,fab} - p_{fa,fb}) = 0.$$
So, out of $21$ possible coordinates only $11$ are essential (i.e., occur as numerators of local coordinates on the quotient) and other $10$ ones are exceed. Therefore, the quotient is a locally closed subset in $\mathbb{P}^{10}$ obtained by intersection of nonzero loci of $p_{f,fa}$ and $p_{f,fb}$ with a closed subset defined by the above equations.
In particular, $\widehat{Y}$ is a complete intersection.
}
\end{example}

\end{document}